\documentclass[preprint,12pt]{elsarticle}

\usepackage{amssymb}
\usepackage{amsmath, amsthm, graphicx, xcolor}

\journal{Applied Mathematics Letters}

\newtheorem{Thm}{Theorem}[section]
\newtheorem{Lemma}{Lemma}[section]

\newtheorem{Remark}{Remark}[section]

\def\m{\mathbb}					\def\p{\partial}		     
					\def\lam{\lambda}		
		    	\def\v{\varepsilon}        
\def\a{\alpha}      	    	\def\b{\beta}

\begin{document}

\begin{frontmatter}

\title{Revisit to Fritz John's paper on the blow-up of nonlinear wave equations\tnoteref{a}}
\tnotetext[a]{The original paper is "Blow-up of solutions of nonlinear wave equations in three space dimensions, Manusctipta Mathematica, 28, 235-268 (1979)". See \cite{John}.}


\author{Xin Yang\corref{cor1}}
\ead{yangxin1@msu.edu}
\cortext[cor1]{Corresponding author}

\author{Zhengfang Zhou\corref{cor2}}
\ead{zfzhou@math.msu.edu}

\address{Department of Mathematics, Michigan State University, East Lansing, MI 48824, USA}
%\address[xin]{Department of Mathematics, Michigan State University, East Lansing, MI 48824, USA}
%\address[zf]{Department of Mathematics, Michigan State University, East Lansing, MI 48824, USA}

\begin{abstract}
In Fritz John's famous paper (1979), he discovered that for the wave equation $\Box u=|u|^p$, where $1<p<1+\sqrt{2}$ and $\Box$ denoting the d'Alembertian, there is no global solution for any nontrivial and compactly supported initial data. This paper is intended to simplify his proof by applying a Gronwall's type inequality.

\end{abstract}


\end{frontmatter}


\section{Introduction}
In 1979, Fritz John published his pioneering work \cite{John}, which was the first one that discovered the critical power of the blow-up phenomenon for wave equations. After this article, many people worked on this kind of blow-up problem. For details and many other related references, see \cite{Sideris, Glassey 1, Glassey 2, Strauss, GLS, LS, Kato, LZ, JZ, Schaeffer, YZ, Zhou}. These work generalize the critical power to other dimensions, some of them also provide simpler proof by imposing additional assumptions or by applying different methods.
 
The paper \cite{John} claims the following well-known Theorem.
\begin{Thm}\label{mainthm}
Let $\phi:\m{R}\rightarrow\m{R}$ be a continuous function which satisfies \[\phi(0)=0,\quad \limsup_{s\rightarrow 0}\,\phi(s)/|s|<\infty.\] Moreover, suppose there exists $A>0$ and $1<p<1+\sqrt{2}$ such that for all $s\in\m{R}, \phi(s)\geq A|s|^{p}$. Then for any function $u\in C^2\big(\m{R}^{3}\times[0,\infty)\big)$ that solves
\begin{eqnarray}
\left\{\begin{array}{lll}\label{Prob}
\Box u(x,t)=\phi\big(u(x,t)\big) &\quad\text{for}\quad & x\in\m{R}^{3},\; t\in(0,\infty), \\
u(x,0)=f(x) &\quad\text{for}\quad & x\in\m{R}^3,\\
u_{t}(x,0)=g(x) &\quad\text{for}\quad & x\in\m{R}^3,
\end{array}\right.
\end{eqnarray}
where $f\in C^{3}(\m{R}^{3}), g\in C^{2}(\m{R}^{3})$ and both of them have compact support, we have
$u\equiv 0$ in $\m{R}^{3}\times[0,\infty)$. Here $\Box$ denotes the d'Alembertian operator: $\Box=\frac{\p^2}{\p t^2}-
\sum\limits_{i=1}^{3}\frac{\p^2}{\p x_{i}^2}$.
\end{Thm} 

The key step to prove this Theorem is the statement as following.

\begin{Thm}\label{keythm}
Let $A>0$, $1<p<1+\sqrt{2}$, let $u$ be a $C^{2}\big(\m{R}^3\times[0,\infty)\big)$ solution of 
$\Box u\geq A|u|^{p}$.  Moreover, suppose there exists a point $(x^{0},t_{0})\in\m{R}^{4}$ such that 
$u^{0}(x,t)\geq 0$ for $(x,t)\in\Gamma^{+}(x^{0},t_{0})$, then $u$ has compact support and 
$\text{supp}\; u\subset\Gamma^{-}(x^{0},t_{0}).$
\end{Thm}

\begin{Remark}\label{Decomp. of u}
If $u$ satisfies $\Box u=w$ with initial data $f$ and $g$, then one decomposes $u$ by $u=u^0+u^1$, where $u^{0}$ solves
\begin{eqnarray}\left\{\begin{array}{lll}
\Box u^0(x,t)=0 &\quad\text{for}\quad & x\in\m{R}^{3},\; t\in(0,\infty), \\
u^0(x,0)=f(x) &\quad\text{for}\quad & x\in\m{R}^3,\\
u^0_{t}(x,0)=g(x) &\quad\text{for}\quad & x\in\m{R}^3,
\end{array}\right.
\end{eqnarray}
and $u^{1}$ suffices
\begin{eqnarray}\left\{\begin{array}{lll}
\Box u^1(x,t)=w(x,t) &\quad\text{for}\quad & x\in\m{R}^{3},\; t\in(0,\infty), \\
u^{1}(x,0)=u^{1}_{t}(x,0)=0 &\quad\text{for}\quad & x\in\m{R}^3.
\end{array}\right.
\end{eqnarray} 

\end{Remark}

\begin{Remark}
For any $(x^{0},t_{0})\in\m{R}^{4}$, the forward cone $\Gamma^{+}(x^{0},t_{0})$ and the backward cone $\Gamma^{-}(x^{0},t_{0})$ are defined as $\,\Gamma^{+}(x^{0},t_{0})
=\{(x,t):|x-x^{0}|\leq t-t_{0}, t\geq 0\}$ and $\,\Gamma^{-}(x^{0},t_{0})
=\{(x,t):|x-x^{0}|\leq t_{0}-t, t\geq 0\}$.
\end{Remark}

In the proof of Theorem \ref{keythm}, \cite{John} employs a technical induction which requires complicated calculations. By introducing a suitable nonlinear functional, this paper gives a much more succinct proof which follows a Gronwall's type inequality.

The organization of this paper is as following: In Section \ref{Preliminaries}, it is shown how Theorem \ref{keythm} implies Theorem \ref{mainthm}, the argument is from \cite{John}. In addition, some notations and a basic Lemma are introduced, where the Lemma is the key technique used in Section \ref{Using G's type ineq.}. In Section \ref{proof of keythm}, we prove Theorem \ref{keythm}. More precisely, Section \ref{Set-up steps}, the first part of the proof, follows from \cite{John} with modifications while Section \ref{Using G's type ineq.}, the rest part of the proof, comes from our own observation. 

\section{Preliminaries}\label{Preliminaries}
\subsection{Theorem \ref{keythm} implies Theorem \ref{mainthm}}
\begin{proof}
Firstly, one can assume that both $f$ and $g$ have support in  $B({\bf 0},\rho)\triangleq \{x\in\m{R}^{3}:|x|<\rho\}$, then by Huygens' principle, $u^0\equiv 0$ in $\Gamma^{+}({\bf 0},\rho)$.  It follows from Theorem \ref{keythm} that $\text{supp}\;\,u\subset\Gamma^{-}({\bf 0},\rho)$. Secondly, one considers the function $v(x,t)\triangleq u(x,\rho-t)$ for $x\in\m{R}^{3}, 0\leq t\leq \rho$. Using the assumptions on $\phi$ in Theorem \ref{mainthm}, one can see $|\Box v|\leq M|v|$ for some fixed $M$ depending on $u$ and $\phi$. Then by energy estimate, $v\equiv 0$ in  $\Gamma^{-}({\bf 0},\rho)$. Thus Theorem \ref{mainthm} is verified.
\end{proof}

\subsection{Some Notations}

For any function $h:\m{R}^3\times[0,\infty)\rightarrow\m{R}$, its radial average function (with respect to spatial variable) $\bar{h}$ is defined by
\begin{equation}\label{average}
\bar{h}(r,t)=\frac{1}{4\pi}\int_{|\xi|=1}h(r\xi,t)\,dS_{\xi},\quad\forall\,(r,t)\in [0,\infty)\times[0,\infty).
\end{equation}

Now let's calculate the radial average of the solution $u$ to the wave equation $\Box u=w$ with zero initial data. To do so, one defines  $v:[0,\infty)\times[0,\infty)\rightarrow\m{R}$ by $v(r,t)=r\,\bar{u}(r,t)$, then we will have
\[\left\{\begin{array}{lll}
\Box v(r,t)=r\,\overline{w}(r,t) &\quad\text{for}\quad & r\in[0,\infty),\;t\in[0,\infty),\\
v(r,0)=0,\;v_{t}(r,0)=0 &\quad\text{for}\quad & r\in[0,\infty).
\end{array}\right.\]
Hence 
\[v(r,t)=\frac{1}{2}\int_{0}^{t}\int_{|r-t+s|}^{r+t-s}\lambda\,\overline{w}(\lambda,s)\,d\lambda\,ds,\]
which implies
\begin{equation}\label{radial ave. of inhomo. soln}
\bar{u}(r,t)=\frac{1}{2r}\int_{0}^{t}\int_{|r-t+s|}^{r+t-s}\lambda\,\overline{w}(\lambda,s)\,d\lambda\,ds.
\end{equation}

For convenience, one defines the operator $P$ acting on $\sigma(r,t)$ with domain $[0,\infty)\times[0,\infty)$ by 
\begin{equation}\label{P operator}
P\sigma(r,t)=\iint\limits_{R_{r,t}}\frac{\lam}{2r}\,\sigma(\lam,s)\,d\lam\,ds,\quad (r,t)\in [0,\infty)\times[0,\infty),
\end{equation}
where $R_{r,t}=\{(\lam,s):0\leq s\leq t, |r-t+s|\leq \lam \leq r+t-s\}$ (See Figure \ref{fig first}).
\begin{figure}[htbp]
\centering
\includegraphics[width=0.75\linewidth]{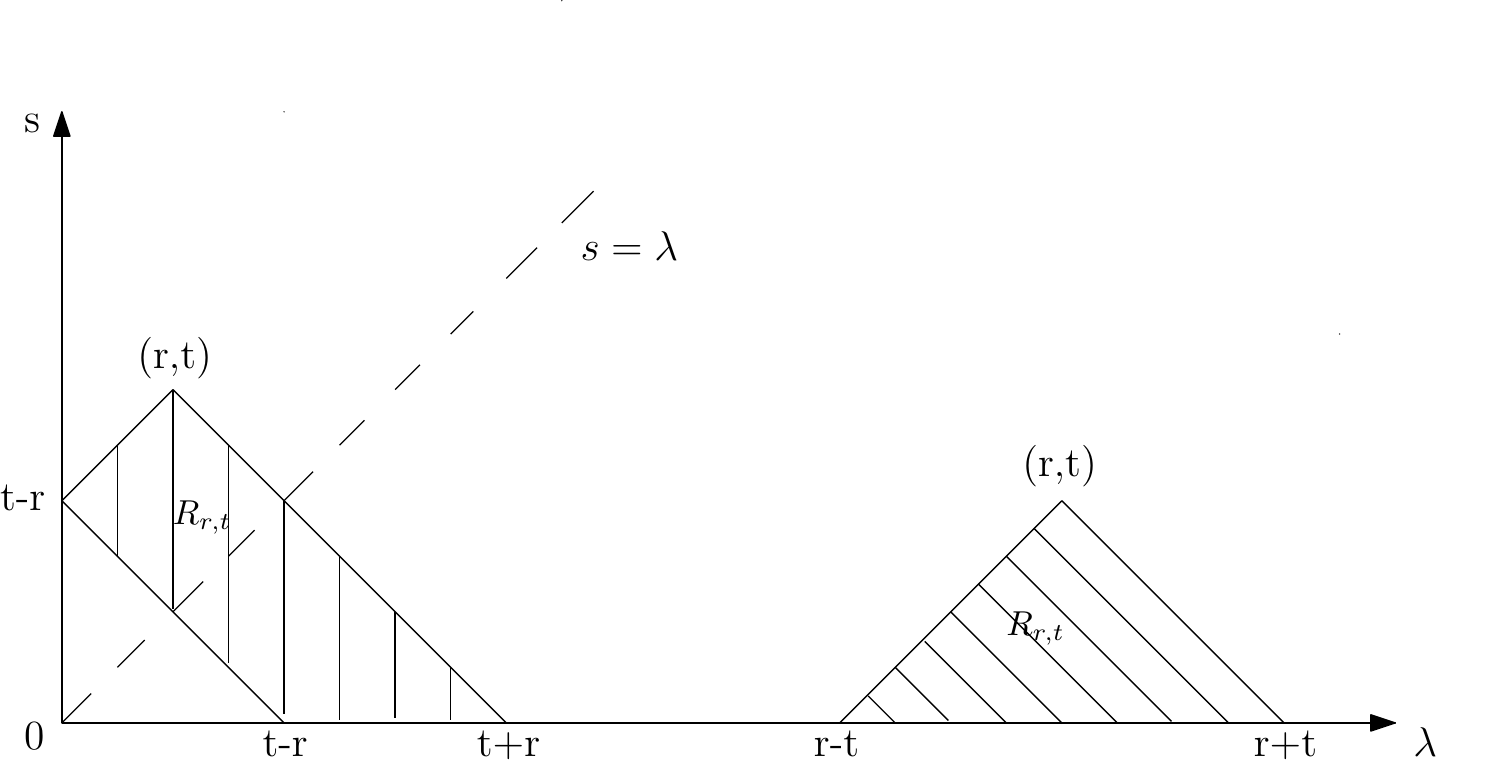}
\caption{$R_{r,t}$}
\label{fig first}
\end{figure}
It is clear that $P$ is a positive operator and now we can rewrite (\ref{radial ave. of inhomo. soln}) as 
\begin{equation}\label{inhomo. soln. rep. by P}
\bar{u}(r,t)=P\overline{w}(r,t),\quad (r,t)\in[0,\infty)\times[0,\infty),
\end{equation}
whenever $u$ solves $\Box u=w$ with initial data.

\subsection{A basic Lemma}
Next, we introduce a  Lemma which is a generalized Gronwall's type inequality with weight.

\begin{Lemma}\label{intineqlemma}
Let $t_0,t_1\in\m{R}$ with $t_0\leq t_1$, suppose $H:[t_0,\infty)\rightarrow[0,\infty)$ to be continuous and $H(r)>0$ for any
$r>t_1$. Then there does not exist constants $C>0, a>1, b\geq -1$ such that 
\begin{equation}\label{intineq}
H(r)\geq C\int_{t_1}^{r}H^{a}(\a)(\a-t_0)^{b}\,d\a,\qquad\forall\, r\geq t_1.
\end{equation}
\end{Lemma}

\begin{proof}
Suppose there exist $C>0, a>1, b\geq -1$ such that (\ref{intineq}) holds, then one defines $J:[t_1+1,\infty)\rightarrow(0,\infty)$ by \[J(r)=\int_{t_1}^{r}H^{a}(\a)(\a-t_0)^{b}\,d\a.\]
Now for any $r\geq t_1+1$, it follows from (\ref{intineq}) that $0<J(r)\leq \frac{H(r)}{C}$ and therefore
\[J'(r)=H^{a}(r)(r-t_0)^{b}\geq C^{a}J^{a}(r)(r-t_0)^{b}.\]
As a result, for any $r_0>t_1+1$,
\[\int_{t_1+1}^{r_0}\frac{J'(r)}{J^{a}(r)}\,dr\geq C^{a}\int_{t_1+1}^{r_0}(r-t_0)^{b}\,dr.\]
Now the left hand side is bounded by 
$$\frac{J^{1-a}(t_1+1)}{a-1} $$
which is a fixed number. However, the Right Hand Side $\rightarrow\infty$ when $r_0\rightarrow\infty$, since 
$b\geq -1$.  Thus, the Lemma follows.
\end{proof}
Later in Section \ref{Using G's type ineq.}, in order to prove the finite time blow-up, the goal is to construct a function $H$, associated with the solution of (\ref{Prob}), which satisfies (\ref{intineq}) with $a,b$ related to the exponent $p$.

\section{Proof of Theorem \ref{keythm}}
\label{proof of keythm}
\subsection{Set-up steps}
\label{Set-up steps}

To verify Theorem \ref{keythm}, first of all, without loss of generality, one can assume $x_0={\bf 0}$ to be the origin in  $\m{R}^{3}$, otherwise just doing a translation. In addition, we can suppose $A=1$, otherwise just doing a dilation.

Now using proof by contradiction, one assumes $\text{supp}\; u$ is not in $\Gamma^{-}({\bf 0},t_{0})$, then there exists $(x^{1},t_{1})\notin\Gamma^{-}({\bf 0},t_{0})$ but $u(x^{1},t_{1})\neq 0$. Set $t_{2}=t_{1}+|x^{1}|$, 
then $t_2>t_0$ and therefore $({\bf 0},t_2)\in\Gamma^{+}({\bf 0},t_{0})$. 

Since $u^{0}\geq 0$ in $\Gamma^{+}({\bf 0},t_{0})$ and $({\bf 0},t_2)\in\Gamma^{+}({\bf 0},t_{0})$, then $u^{0}\geq 0$ in $\Gamma^{+}({\bf 0},t_{2})$. As a consequence, for any $0\leq r\leq t-t_2$,
\begin{equation}
\bar{u}(r,t)=\overline{u^0}(r,t)+\overline{u^1}(r,t)\geq \overline{u^1}(r,t).
\end{equation}
Noticing the fact that  $P$ is positive and the assumption $\Box u^1\geq |u|^{p}$ in Theorem \ref{keythm}, it follows from (\ref{inhomo. soln. rep. by P}) that
\begin{equation}\label{rad. ave. point est. 1}
\bar{u}(r,t)\geq\overline{u^1}(r,t)\geq P\big(\overline{|u|^{p}}\big)(r,t)=\iint\limits_{R_{r,t}}\frac{\lam}{2r}\,\overline{|u|^{p}}(\lam,s)\,d\lam\,ds.
\end{equation}
Because of the simple fact $\overline{|u|^{p}}(\lam,s)\geq |\bar{u}(\lam,s)|^{p}$, one has that

\begin{equation}\label{rad. ave. point est. 2}
\bar{u}(r,t)\geq \iint\limits_{R_{r,t}}
\frac{\lam}{2r}|\bar{u}(\lam,s)|^p\,d\lam\,ds, \quad\forall\,0\leq r\leq t-t_2.
\end{equation}
From (\ref{rad. ave. point est. 1}), we have 
\begin{equation}\label{positive on s=t_2+lambda}
\bar{u}(\delta, t_2+\delta)>0, \quad\forall\,\delta>0,
\end{equation}
since the point $(|x^{1}|,t_{1})$ lies in $R_{\delta,t_2+\delta}$ and $u(x^{1},t_1)\neq 0$.

Now one fixes $t_2$ and a positive number $\delta$, then considers the regions (See Figure \ref{fig second})
\begin{figure}[htbp]
\centering
\includegraphics[width=0.8\linewidth]{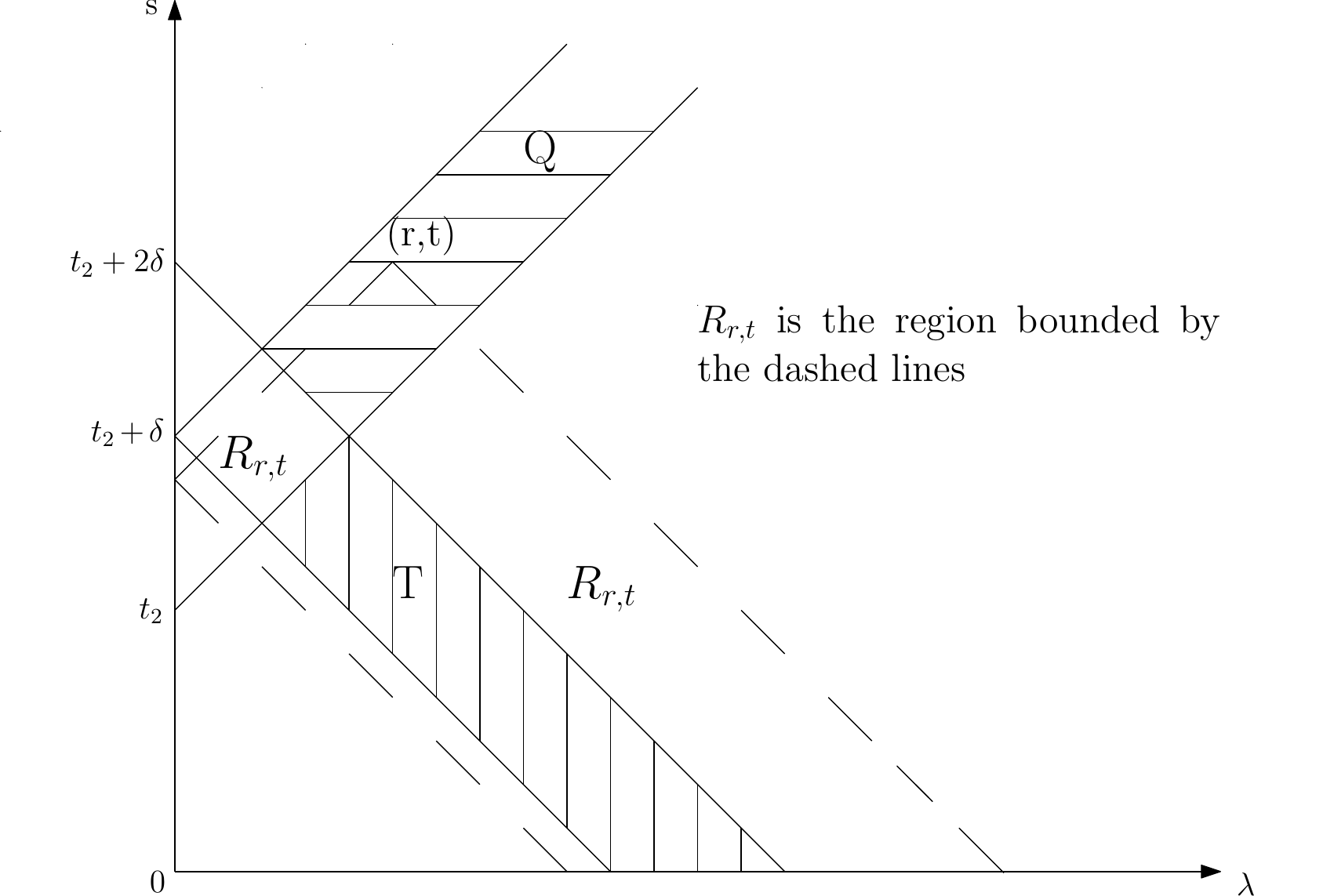}
\caption{$T$ and $Q$}
\label{fig second}
\end{figure}
\[T=\{(\lam,s):t_2+\delta\leq s+\lam\leq t_2+2\delta, s-\lam\leq t_2, s\geq 0\},\]
\[Q=\{(\lam,s):t_2+2\delta\leq s+\lam, t_2\leq s-\lam\leq t_2+\delta\}.\]

It is easy to check that the fixed region $T\subset R_{r,t}$ for any $(r,t)\in Q$. Then it follows from (\ref{rad. ave. point est. 2}) that for any $(r,t)\in Q$,
\begin{equation}\label{averpointest}
\bar{u}(r,t)\geq \iint\limits_{T}\frac{\lam}{2r}|\bar{u}(\lam,s)|^{p}\,d\lam\,ds=\frac{M}{r},
\end{equation}
where $M$ is a positive constant due to (\ref{positive on s=t_2+lambda}).

Let $\Sigma=\{(r,t):0\leq r\leq t- t^{*}\}$, where $t^{*}\triangleq t_2+2\delta$. For any $(r,t)\in\Sigma$, one defines the sets (See Figure \ref{fig third})
\begin{figure}[htbp]
\centering
\includegraphics*[width=1\linewidth]{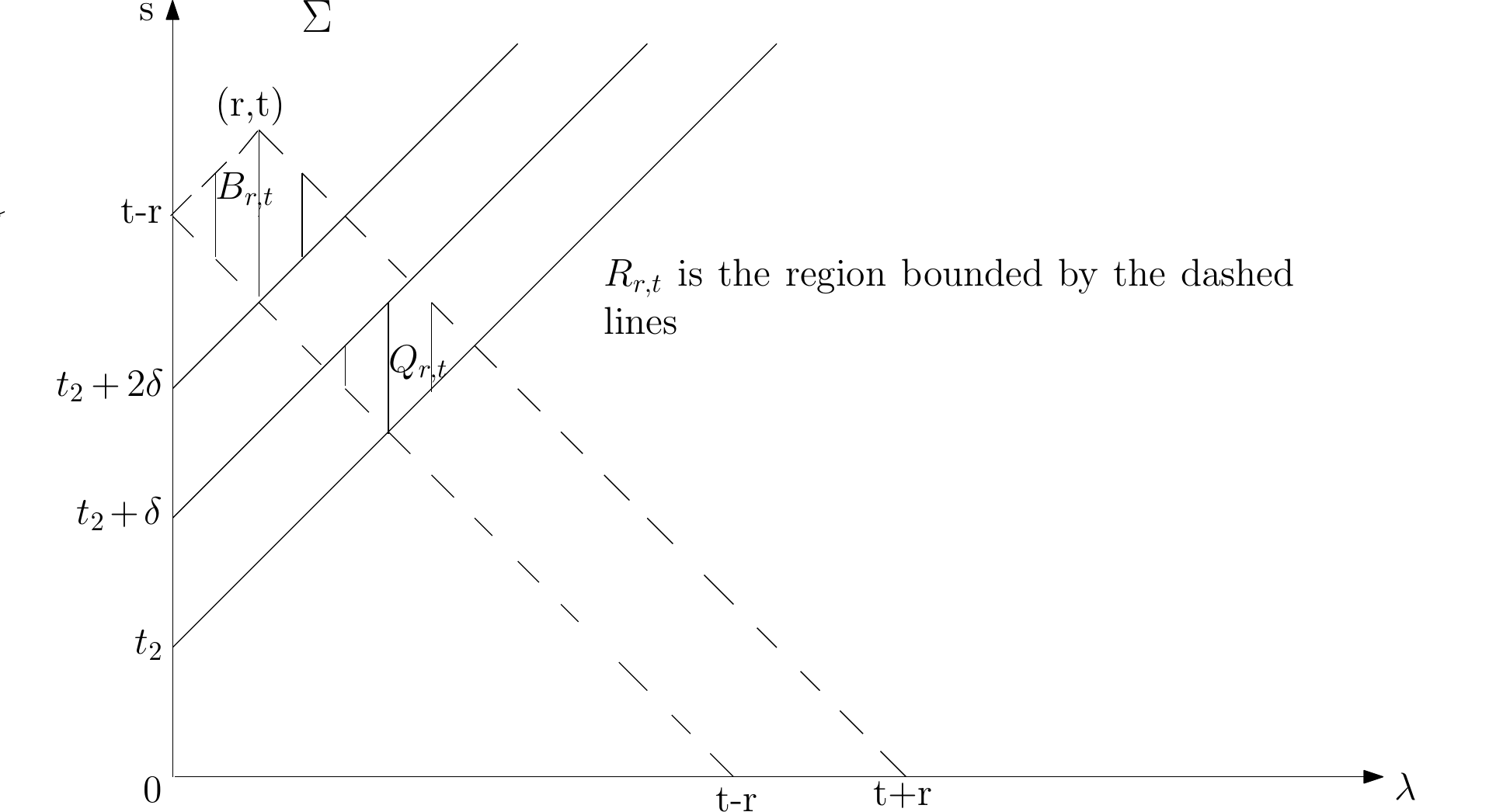}
\caption{$Q_{r,t}$ and $B_{r,t}$}
\label{fig third}
\end{figure}
\[Q_{r,t}=\{(\lam,s):t-r\leq \lam+s\leq t+r,\; t_2\leq s-\lam\leq t_2+\delta\},\]
\[B_{r,t}=\{(\lam,s):t-r\leq \lam+s\leq t+r,\; t_2+2\delta\leq s-\lam\leq t-r\}.\]

Then for any $(r,t)\in\Sigma$,  we get
\begin{equation}\label{graphinfo3}
Q_{r,t}\subset R_{r,t},\;B_{r,t}\subset R_{r,t},\;Q_{r,t}\subset Q,\;B_{r,t}\subset \Sigma.
\end{equation}
Thus from (\ref{rad. ave. point est. 2}), (\ref{averpointest}), (\ref{graphinfo3}), one obtains
\begin{eqnarray*}
\bar{u}(r,t) &\geq & \iint\limits_{Q_{r,t}}\frac{\lam}{2r}|\bar{u}(\lam,s)|^{p}\,d\lam\,ds\\
&\geq & \iint\limits_{Q_{r,t}}\frac{\lam}{2r}\frac{M^p}{\lam^p}\,d\lam\,ds\\
&=& \frac{M^p}{2r}\iint\limits_{Q_{r,t}}\lam^{1-p}\,d\lam\,ds.
\end{eqnarray*}
The area of $Q_{r,t}$ is $r\delta$ and for any point $(\lambda,s)$ in it, $\lam\leq t+r$. As a result, 
\begin{equation}\label{betteraverpointest}
\bar{u}(r,t)\geq \frac{M^p}{2r}\, r\delta\, (t+r)^{1-p}=
C_0(t+r)^{1-p}, \quad\forall\,(r,t)\in\Sigma,
\end{equation}
where $C_0=M^{p}\delta/2$ is a positive constant.

\subsection{Using Gronwall's type inequality}
\label{Using G's type ineq.}
Until now, the ideas are natural and all the estimates are not difficult to get. But after this step, \cite{John} claims an induction:
\[\bar{u}(r,t)\geq C_{k}(t+r)^{-1}(t-r- t^{*})^{a_{k}}(t-r)^{-b_{k}},\quad\forall\,(r,t)\in\Sigma,\quad \forall\, k\geq 1,\]
where $a_{k}, b_{k}, c_{k}$  satisfy complex recurrence formulas. By tedious computations, one finds that $(t-r- t^{*})^{a_{k}}$ is the dominant term and $a_{k}\rightarrow\infty$, then the blow-up follows when taking $(t-r- t^{*})$ to be a fixed large number and $k\rightarrow\infty$.

In the following, we will carry out a much more concise argument by introducing a suitable nonlinear functional and making use of Lemma \ref{intineqlemma}.

Firstly, we observe from (\ref{rad. ave. point est. 2}) and (\ref{graphinfo3}) that for any $(r,t)\in\Sigma$, 
\begin{equation}\label{ineq 1}
\bar{u}(r,t)\geq \iint\limits_{R_{r,t}}\frac{\lam}{2r}|\bar{u}(\lam,s)|^{p}\,d\lam\,ds
\geq \iint\limits_{B_{r,t}}\frac{\lam}{2r}|\bar{u}(\lam,s)|^{p}\,d\lam\,ds,
\end{equation}
which implies $\bar{u}\geq 0$ on $\Sigma$ and especially by (\ref{graphinfo3}), $\bar{u}\geq 0$ on $B_{r,t}$. So we can remove the absolute value sign in the last term of (\ref{ineq 1}) to get
\begin{equation}\label{posaverineq}
\bar{u}(r,t)\geq \iint\limits_{B_{r,t}}\frac{\lam}{2r}\,\bar{u}^{p}(\lam,s)\,d\lam\,ds,\quad\forall\, (r,t)\in \Sigma.
\end{equation}
By change of variable: $\a=\lam+s$ and $\b=s-\lam$, we obtain
\begin{equation}\label{posaverineqcov}
\bar{u}(r,t)\geq \frac{1}{8r}\int_{t-r}^{t+r}\int_{ t^{*}}^{t-r}(\a-\b)\bar{u}^{p}
\bigg(\frac{\a-\b}{2},\frac{\a+\b}{2}\bigg)\,d\b\,d\a.
\end{equation}

Here comes an important observation by considering $F:\Sigma'\rightarrow\m{R}$, where $F(\a,\b)\triangleq\bar{u}\big(\frac{\a-\b}{2},\frac{\a+\b}{2}\big)$ and $\Sigma'=\{(r,t): t^{*}\leq t\leq r\}$ corresponding to the definition of $\Sigma$. Now 
(\ref{betteraverpointest}) becomes
\begin{equation}\label{pointestcov}
F(r,t)\geq C_0\, r^{1-p}, \quad \forall\, (r,t)\in\Sigma'.
\end{equation}
Moreover, (\ref{posaverineqcov}) becomes 
\begin{equation}\label{pointestintcov}
F(r,t)\geq \frac{1}{4(r-t)}\int_{t}^{r}\int_{ t^{*}}^{t}(\a-\b)F^{p}(\a,\b)\,d\b\,d\a, \quad \forall\, (r,t)\in\Sigma'.
\end{equation}

From here, it attempts to employ the Gronwall's inequality technique, which reduces the blow-up problem to a pure analysis technique. This is our motivation. 

However we can not apply it directly since Gronwall's inequality only deals with single variable and the right hand side being a double integral. In addition, (\ref{pointestintcov}) involves some weight functions. To overcome these difficulties, we introduce some new functions and make use of Lemma \ref{intineqlemma} together with (\ref{pointestcov}).

For $q\geq 1$ to be determined later, we define $G:\Sigma'\rightarrow \m{R}$ by $$G(r,t)=(r-t)^{q}F(r,t).$$ 
From (\ref{pointestintcov}), 
\begin{equation}\label{G1ineq}
G(r,t)\geq\frac{1}{4}(r-t)^{q-1}\int_{t}^{r}\int_{ t^{*}}^{t}G^{p}(\a,\b)(\a-\b)^{1-qp}\,d\b\,d\a.
\end{equation}
We define $H:[ t^{*},\infty)\rightarrow\m{R}$ by $$H(r)=\int_{ t^{*}}^{r}G(r,t)\,dt$$
and integrate (\ref{G1ineq}) for $t$ from $ t^{*}$ to $r$. Then the left hand side of (\ref{G1ineq}) becomes $H(r)$. In order to exploit Lemma \ref{intineqlemma}, the right hand side, hopefully after changing the order of integration, can become a single integral of $H(r)$. We calculate as follows,
\begin{eqnarray*}
H(r) &\geq & \frac{1}{4}\int_{ t^{*}}^{r}\int_{t}^{r}\int_{ t^{*}}^{t}G^p(\a,\b)(\a-\b)^{1-qp}
(r-t)^{q-1}\,d\b\,d\a\,dt\\
&=& \frac{1}{4}\int_{ t^{*}}^{r}\int_{ t^{*}}^{\a}G^p(\a,\b)(\a-\b)^{1-qp}\int_{\b}^{\a}(r-t)^{q-1}
\,dt\,d\b\,d\a\\
&=& \frac{1}{4q}\int_{ t^{*}}^{r}\int_{ t^{*}}^{\a}G^p(\a,\b)(\a-\b)^{1-qp}\big[(r-\b)^q-(r-\a)^q\big]
\,d\b\,d\a.
\end{eqnarray*}
Since $q\geq 1$, then $(r-\b)^q-(r-\a)^q\geq (\a-\b)^q$. As a result,
\begin{equation}
H(r)\geq \frac{1}{4q}\int_{ t^{*}}^{r}\int_{ t^{*}}^{\a}G^p(\a,\b)(\a-\b)^{1-qp+q}\,d\b\,d\a.
\end{equation}
We choose $q=p/(p-1)$ and denote $C$ to be a constant which is independent of the variable $r$ but may be different from line to line, then 
\begin{equation}\label{H1ineq}
H(r)\geq C\int_{ t^{*}}^{r}\int_{ t^{*}}^{\a}
G^p(\a,\b)(\a-\b)^{1-p}\,d\b\,d\a.
\end{equation}
Using Holder's inequality, 
\begin{eqnarray*}
H(\a) &=& \int_{ t^{*}}^{\a}G(\a,\b)\,d\b \\
&=& \int_{ t^{*}}^{\a}G(\a,\b)(\a-\b)^{\frac{1-p}{p}}(\a-\b)^{\frac{p-1}{p}}\,d\b \\
&\leq & \bigg(\int_{ t^{*}}^{\a}G^p(\a,\b)(\a-\b)^{1-p}\,d\b\bigg)^{\frac{1}{p}}
\bigg(\int_{ t^{*}}^{\a}(\a-\b)\,d\b\bigg)^{\frac{p-1}{p}}.
\end{eqnarray*}
So
\begin{align*}
\int_{ t^{*}}^{\a}G^p(\a,\b)(\a-\b)^{1-p}\,d\b &\geq 
H^p(\a)\,\bigg(\int_{ t^{*}}^{\a}(\a-\b)\,d\b\bigg)^{1-p}\\
&= C\,H^p(\a)\,(\a- t^{*})^{2-2p}.
\end{align*}
Plugging in (\ref{H1ineq}) gives 
\begin{equation}\label{Hineqsingle}
H(r)\geq C\,\int_{ t^{*}}^{r}H^p(\a)(\a- t^{*})^{2-2p}\,d\a,\quad\forall r\geq  t^{*}.
\end{equation}

Now it is almost done! Only trouble for applying Lemma \ref{intineqlemma} is that $2-2p$ may be less than $-1$. In order to raise the power of $\a- t^{*}$, we borrow the idea from \cite{Sideris}. Namely, we write $H^p(\a)=H^{1+\v}(\a)\,H^{p-1-\v}(\a)$ and hope to find $H(\a)\geq C (\a- t^{*})^{d}$ for some $d$, which can increase the power of $(\a- t^{*})$ by $(p-1-\v)d$.

From (\ref{pointestcov}) and the definition of $G$ and $H$, we have for any $\a\geq 2 t^{*}$, 
\begin{eqnarray}
H(\a)= \int_{ t^{*}}^{\a}G(\a,\b)\,d\b &\geq & C_0\int_{ t^{*}}^{\a}(\a-\b)^{q}\,\a^{1-p}\,d\b \nonumber \\
&=& \frac{C_0}{q+1}\,\alpha^{1-p}(\alpha-t^{*})^{q+1} \nonumber \\
&\geq & \frac{C_0}{q+1}\,2^{1-p}(\alpha-t^{*})^{1-p}(\alpha-t^{*})^{q+1} \nonumber\\
&=& C\,(\a- t^{*})^{2-p+q}.
\end{eqnarray}
Getting back to (\ref{Hineqsingle}), for any $r\geq 2 t^{*}$,  one gets 
\begin{align*}
H(r) &\geq C\int_{2 t^{*}}^{r}H^{1+\v}(\a)\,H^{p-1-\v}(\a)\,(\a- t^{*})^{2-2p}\,d\a \\
&\geq C\int_{2 t^{*}}^{r}H^{1+\v}(\a)\,(\a- t^{*})^{s(p,\v)}\,d\a, 
\end{align*}
where 
\begin{align*}
s(p,\v)&=(p-1-\v)(2-p+q)+2-2p\\
&=-p^2+2p-\v\bigg(2-p+\dfrac{p}{p-1}\bigg).
\end{align*}

Since $1<p<1+\sqrt{2}$, then $-p^2+2p>-1$ and
therefore it is possible to choose small $\v$ in $(0,p-1)$ such that $s(p,\v)\geq -1$.

Now applying Lemma \ref{intineqlemma} with $t_0= t^{*},\,t_1=2 t^{*},\,a=1+\v,\,b=s(p,\v)$, we get contradiction. Thus we reject the assumption $\text{supp}\; u \nsubseteq\Gamma^{-}({\bf 0},t_{0})$, hence Theorem \ref{keythm} follows.

\section*{Acknowledgements}
The authors thank the referees for their suggestions which make this work more clearly.





\end{document}